\documentclass[a4paper]{article}
\usepackage[british]{babel}
\usepackage[utf8]{inputenc}
\usepackage{csquotes}
\usepackage[a4paper, total={5.5in, 8.5in}]{geometry}

\usepackage{amsmath,amssymb,amsthm}
\usepackage{dsfont}
\usepackage{tikz}
\usepackage{tikz-cd}

\usepackage{algorithm}
\usepackage{algpseudocode}
\usepackage{hyperref}
\usepackage{cleveref}
\usepackage{float}
\usepackage{enumerate}

\title{The Reidemeister spectra of low dimensional almost-crystallographic groups}
\author{Sam Tertooy\thanks{Supported by long term structural funding -- Methusalem grant of the Flemish Government. }}
\date{\today}

\theoremstyle{plain}
\newtheorem{theorem}{Theorem}[section]
\newtheorem{lemma}[theorem]{Lemma}
\newtheorem{proposition}[theorem]{Proposition}
\newtheorem{corollary}[theorem]{Corollary}

\theoremstyle{definition}
\newtheorem{definition}[theorem]{Definition}

\theoremstyle{remark}

\DeclareMathOperator{\Aut}{Aut}

\DeclareMathOperator{\Aff}{Aff}

\DeclareMathOperator{\GL}{GL}

\DeclareMathOperator{\tr}{tr}

\newcommand{\I}{\mathds{1}}

\DeclareMathOperator{\Spec}{Spec}

\DeclareMathOperator{\Fix}{Fix}

\newcommand{\NN}{\mathbb{N}}
\newcommand{\ZZ}{\mathbb{Z}}

\newcommand{\RR}{\mathbb{R}}

\newcommand{\lie}[1]{{\mathfrak {#1}}}
\newcommand{\comm}[2]{\left[#1,#2\right]} 

\newenvironment{smallpmatrix}
{\left(\begin{smallmatrix}}
	{\end{smallmatrix}\right)}

\newcommand*\Let[2]{\State #1 \(:=\) #2}

\begin{document}
	
	\maketitle
	
	\begin{center}
		This is an Accepted Manuscript of an article published by Taylor \& Francis in Experimental Mathematics on 11 Jul 2019, available online:  \href{https://doi.org/10.1080/10586458.2019.1636426}{https://www.tandfonline.com/10.1080/10586458.2019.1636426}.
	\end{center}

\begin{abstract}
We determine which non-crystallographic, almost-crystallographic groups of dimension 4 have the \(R_\infty\)-property. We then calculate the Reidemeister spectra of the 3-dimensional almost-crystallographic groups and the 4-dimensional almost-Bieberbach groups.
\end{abstract}

\section{Introduction}
Let \(G\) be any group and \(\varphi: G \to G\) an endomorphism of this group. Define an equivalence relation \(\sim_\varphi\) on \(G\) given by
\begin{equation*}
\forall g,g' \in G: g \sim_\varphi g' \iff \exists h \in G: g = hg'\varphi(h)^{-1}.
\end{equation*}
An equivalence class \([g]_\varphi\) is called a \emph{Reidemeister class} of \(\varphi\) or \(\varphi\)-\emph{twisted conjugacy class}. The \emph{Reidemeister number} \(R(\varphi)\) is the number of Reidemeister classes of \(\varphi\) and is therefore always a positive integer or infinity. The \emph{Reidemeister spectrum} of a group \(G\) is the set of all Reidemeister numbers when considering  all possible automorphisms of that group:
	\begin{equation*}
	\Spec_R(G) := \{R(\varphi) \mid \varphi \in \Aut(G)\}.
	\end{equation*}
	If \(\Spec_R(G) = \{\infty\}\) we say that \(G\) has the \emph{\(R_\infty\)-property}.
	
Reidemeister numbers originate in Nielsen fixed point theory, where they are defined as the number of fixed point classes of a self-map of a topological space \cite{jian83-1}, although they also yield applications in algebraic geometry and representation theory \cite{ft15-1}.

It turns out that many (infinite) groups admit the \(R_\infty\)-property.  This is also the case for most almost-crystallographic groups, e.g. in \cite{dp11-1} it was shown that \(207\) of the \(219\) 3-dimensional crystallographic groups and \(15\) of the \(17\) families of 3-dimensional (non-crystallographic) almost-crystallographic groups all have the \(R_\infty\)-property. Furthermore, in \cite{dkt17-2} it was shown that \(4692\) of the \(4783\) 4-dimensional crystallographic groups admit the \(R_\infty\) property. Moreover, the Reidemeister spectra of all crystallographic groups of dimensions \(1\), \(2\) and \(3\) were calculated, as well as the spectra of the \(4\)-dimensional Bieberbach groups. In this paper we extend these results by studying the 4-dimensional almost-crystallographic groups.

This paper is structured as follows. In the next two sections, we provide the necessary preliminaries on Reidemeister numbers and almost-crystallographic groups. In section 4 we determine which almost-crystallographic groups of dimension \(4\) possess the \(R_\infty\)-property. Sections 5 and 6 are devoted to calculating the Reidemeister spectra of the 3-dimensional almost-crystallographic groups and the 4-dimensional almost-Bieberbach groups respectively. The final section summarises the obtained results.

\section{Reidemeister numbers and spectra}
In this section we introduce basic notions concerning the Reidemeister number. For a general reference on Reidemeister numbers and their connection to fixed point theory, we refer the reader to  \cite{jian83-1}.

The definitions of the Reidemeister number and Reidemeister spectrum were given in the introduction. However, nothing was said on how we actually determine whether a group has the \(R_\infty\)-property, and if not, how we calculate its Reidemeister spectrum. The following lemma is an essential tool for the former.

\begin{lemma}[see {\cite[Section 2.2]{ft15-1}}, {\cite[Lemma 1.1]{gw09-2}}]
	\label{lem:diagram3props}
	Let \(N\) be a normal subgroup of a group \(G\) and \(\varphi \in \Aut(G)\) with \(\varphi(N) = N\). We denote the restriction of \(\varphi\) to \(N\) by \(\varphi|_N\), and the induced automorphism on the quotient \(G/N\) by \(\varphi'\). We then get the following commutative diagram with exact rows:
	\begin{center}
		\begin{tikzcd}
		1 \arrow{r} & N \arrow{r}{}\arrow{d}{\varphi|_N} & G                 \arrow{r}{}\arrow{d}{\varphi}      & G/N           \arrow{r}\arrow{d}{\varphi'} & 1 \\
		1 \arrow{r} & N \arrow{r}{}                & G             \arrow{r}{}      & G/N        \arrow{r} & 1
		\end{tikzcd}
	\end{center}
	We obtain the following properties:\begin{enumerate}[(1)]
		\item \(R(\varphi) \geq R(\varphi')\),
		\item if \(R(\varphi') < \infty\), \(R(\varphi|_N) = \infty\) and \(| \Fix(\varphi')| < \infty\), then \(R(\varphi) = \infty\).
	\end{enumerate}
\end{lemma}

A direct consequence for characteristic subgroups is the following:

\begin{corollary}
	\label{cor:charRoo}
	Let \(N\) be a characteristic subgroup of \(G\). If either
	\begin{enumerate}[(1)]
		\item the quotient \(G/N\) has the \(R_\infty\)-property, or
		\item \(N\) has finite index in \(G\) and has the \(R_\infty\)-property,
	\end{enumerate}
	then \(G\) has the \(R_\infty\)-property as well.
\end{corollary}

\section{Almost-crystallographic groups}
Let \(G\) be a connected, simply connected, nilpotent Lie group with automorphism group \(\Aut(G)\). The affine group \(\Aff(G)\) is the semi-direct product \(\Aff(G) = G \rtimes \Aut(G)\), where multiplication is defined by \((d_1,D_1)(d_2,D_2) = (d_1D_1(d_2),D_1D_2)\). 
If \(C\) is a maximal compact subgroup of \(\Aut(G)\), then \(G \rtimes C\) is a subgroup of \(\Aff(G)\). A cocompact discrete subgroup \(\Gamma\) of \(G \rtimes C\) is called an \emph{almost-crystallographic group} modelled on the Lie group \(G\). The dimension of \(\Gamma\) is defined as the dimension of \(G\).

If \(\Gamma\) is torsion-free, then it is called an \emph{almost-Bieberbach group}. If \(G = \RR^n\), then it is called a \emph{crystallographic group}, or a \emph{Bieberbach group} if it also torsion-free. 

Crystallographic groups were historically studied first, and are well understood by the three Bieberbach theorems. These theorems have since been generalised to al\-most-crys\-tal\-lo\-gra\-phic groups, which we will briefly discuss below. We refer to \cite{szcz12-1} and \cite{deki96-1} for more information on the original and generalised theorems respectively.

The generalised first Bieberbach theorem says that if \(\Gamma\subseteq \Aff(G)\) is an \(n\)-dimensional almost-crystallographic group, then its \emph{translation subgroup} \(N := \Gamma\cap G\) is a uniform lattice of \(G\) and is of finite index in \(\Gamma\). Moreover, \(N\) is the unique maximal nilpotent normal subgroup of \(\Gamma\), and is therefore characteristic in \(\Gamma\). The quotient group \(F := \Gamma/N\) is a finite group called the \emph{holonomy group} of \(\Gamma\). In fact \(F=\{ A\in \Aut(G)\;|\; \exists a \in G:\, (a,A)  \in \Gamma\}\). If \(\Gamma\) is crystallographic (\(G = \RR^n\)), we may assume that \(N = \ZZ^n\) and \(F\) is a subgroup of \(\GL_n(\ZZ)\).

The generalised second Bieberbach theorem tells us more about automorphisms of al\-most-crys\-tal\-lo\-gra\-phic groups.
 
\begin{theorem}[generalised second Bieberbach theorem]
	\label{thm:gensecbieb}
	Let \(\varphi: \Gamma \to \Gamma\) be an automorphism of an almost-crystallographic group \(\Gamma \subseteq \Aff(G)\) with holonomy group \(F\). Then there exists a \((d,D) \in \Aff(G)\) such that \(\varphi(\gamma) = (d,D) \circ  \gamma \circ (d,D)^{-1}\) for all \(\gamma \in \Gamma\). To shorten notation, we will write \(\varphi = \xi_{(d,D)}\).
\end{theorem}

An automorphism \(\Phi: G \to G\) of a Lie group \(G\) induces an automorphism \(\Phi_*: \lie{g} \to \lie{g}\) of the associated Lie algebra \(\lie{g}\). We will henceforth always denote an induced automorphisms on a Lie algebra with a star (\(*\)) subscript, for example \(A_*\) is the Lie algebra automorphism induced by some \(A \in F\) where \(F \subseteq \Aut(G)\) is the holonomy group of an almost-crystallographic group. In particular, an automorphism \(\varphi = \xi_{(d,D)}\) of an almost-crystallographic group has an associated matrix \(D_*\).

The generalised third Bieberbach theorem is less straightforward to generalise. Unlike for crystallographic groups, it is not true that there are only finitely many \(n\)-dimensional almost-crystallographic groups for a given dimension \(n\). However, we can state that for a given finitely generated torsion-free nilpotent group \(N\), there are (up to isomorphism) only finitely many almost-crystallographic groups \(\Gamma\) such that the translation subgroup of \(\Gamma\) is isomorphic to \(N\). 

In \cite[Section 2.5]{deki96-1}, this generalisation is proved using the concept of an \emph{isolator}, which shall prove useful to us as well.
\begin{definition}
	Let \(G\) be a group with subgroup \(H\). The isolator of \(H\) in \(G\) is defined as
	\begin{equation*}
	\sqrt[G]{H} := \{g \in G \mid g^k \in H \textrm{ for some } k \geq 1\}.
	\end{equation*}
\end{definition}
Although much can be said about isolators, for the purposes of this paper we only care about a very specific result.
\begin{lemma}[see {\cite[Lemma 2.4.2]{deki96-1}}]
	\label{lem:quotientbyisolatorAC}
	Let \(\Gamma\) be an almost-crystallographic group with translation subgroup \(N\) of nilpotency class \(c\). Then the isolator \(\sqrt[N]{\gamma_c(N)} \leq Z(N)\) is a characteristic subgroup of \(\Gamma\). Moreover, the quotient group \(\Gamma / \sqrt[N]{\gamma_c(N)}\) is an almost-crystallographic group whose translation subgroup \(N/\sqrt[N]{\gamma_c(N)}\) has nilpotency class \(c-1\). If \(c = 2\), then this quotient is a crystallographic group.
\end{lemma}

We will now give the most important results for Reidemeister theory applied to almost-crystallographic groups. A first result allows us to easily determine whether an almost-crystallographic group admits the \(R_\infty\)-property or not.

\begin{theorem}[see {\cite[Corollary 3.10]{dp11-1}}]
	\label{thm:det1-AD}
	Let \(\Gamma\) be an \(n\)-di\-men\-sion\-al almost-crystallographic group with holonomy group \(F \subseteq \Aut(G)\) and \(\varphi = \xi_{(d,D)} \in \Aut(\Gamma)\) (where we use the notation of \cref{thm:gensecbieb}). Then
	\begin{align*}
	&R(\varphi) = \infty \\
	&\iff \exists A \in F \text{ such that } \det(\I_n - A_*D_*) = 0\\
	&\iff \exists A \in F \text{ such that } A_*D_* \text{ has eigenvalue }1.
	\end{align*}
\end{theorem}
The second result only holds for almost-Bieberbach groups, and allows for an easy computation of the Reidemeister number of an automorphism.

\begin{theorem}[averaging formula, see {\cite[Theorem 6.11]{hlp12-1}} and {\cite[Theorem 4.3]{ll09-1}}]
	\label{thm:averagingalmostbieb}
	Let \(\Gamma\) be an \(n\)-dimensional almost-Bieberbach group with holonomy group \(F \subseteq \Aut(G)\), and \(\varphi = \xi_{(d,D)} \in \Aut(\Gamma)\) with \(R(\varphi) < \infty\). Then
	\begin{equation*}
	R(\varphi) = \frac{1}{\#F}\sum_{A \in F}|\det(\I_n - A_*D_*)|.
	\end{equation*}
\end{theorem}
In general, this formula does not hold for automorphisms of almost-crystallographic groups, examples can be found in \cite{dkt17-2} and later in this paper. Therefore, the calculation of the Reidemeister spectra usually requires a deeper understanding of how the Reidemeister classes are formed in a specific group.

\section{The \(R_\infty\)-property for 4-dimensional al\-most-crys\-tal\-lo\-graphic groups}
\label{sec:Rinfty4D}
Every almost-crystallographic group of dimension \(1\) or \(2\) is crystallographic. In \cite{dp11-1} it was determined which \(3\)-dimensional almost-crystallographic groups admit the \(R_\infty\)-property. We extend these results to dimension \(4\). In this case the translation subgroup \(N\) is a finitely generated, torsion-free, nilpotent group of rank \(4\) and nilpotency class at most \(3\). Nilpotency class \(1\) is of course the crystallographic case, which was done in \cite{dkt17-2}.

\subsection{Nilpotency class \(2\)}
\label{sec:2stepnilpclass4}
Let \(\Gamma\) be an almost-crystallographic group whose translation subgroup \(N\) is a nilpotent group of rank \(4\) and nilpotency class \(2\). In \cite{deki96-1} it was shown that \(N\) can be given the following presentation:
\begin{equation*}
\left\langle e_1,e_2,e_3,e_4 \;\Bigg\rvert \begin{array}{ll}
\comm{e_2}{e_1} = 1 & \comm{e_3}{e_2} = e_1^{l_1}\\
\comm{e_3}{e_1} = 1 & \comm{e_4}{e_2} = e_1^{l_2}\\
\comm{e_4}{e_1} = 1 & \comm{e_4}{e_3} = e_1^{l_3}
\end{array}\right\rangle.
\end{equation*}
Moreover, let \(G\) be the Lie group that \(\Gamma\) is modelled on. By \cite[Theorem 4.1]{deki95-1}, there exists a faithful affine representation \(\lambda: G \rtimes \Aut(G) \to \Aff(\RR^4)\) such that its restriction to \(\Gamma\) is again a faithful affine representation. In particular, 
\begin{align*}
\lambda(e_1) &= \begin{pmatrix}
1 & 0 & 0 & 0 & 1\\
0 & 1 & 0 & 0 & 0\\
0 & 0 & 1 & 0 & 0\\
0 & 0 & 0 & 1 & 0\\
0 & 0 & 0 & 0 & 1
\end{pmatrix},&\quad
\lambda(e_2) &= \begin{pmatrix}
1 & 0 & -\frac{l_1}{2} & -\frac{l_2}{2} &0\\
0 & 1 & 0 & 0 & 1\\
0 & 0 & 1 & 0 & 0\\
0 & 0 & 0 & 1 & 0\\
0 & 0 & 0 & 0 & 1
\end{pmatrix},\\&&&\\
\lambda(e_3) &= \begin{pmatrix}
1 & \frac{l_1}{2}& 0 & -\frac{l_3}{2} &0\\
0 & 1 & 0 & 0 & 0\\
0 & 0 & 1 & 0 & 1\\
0 & 0 & 0 & 1 & 0\\
0 & 0 & 0 & 0 & 1
\end{pmatrix},&\quad
\lambda(e_4) &= \begin{pmatrix}
1 & \frac{l_2}{2} & \frac{l_3}{2} & 0 & 0\\
0 & 1 & 0 & 0 & 0\\
0 & 0 & 1 & 0 & 0\\
0 & 0 & 0 & 1 & 1\\
0 & 0 & 0 & 0 & 1
\end{pmatrix},
\end{align*}
where the values of \(l_1\), \(l_2\) and \(l_3\) are determined by the relations \([e_3,e_2] = e_1^{l_1}\), \([e_4,e_2] = e_1^{l_2}\) and \([e_4,e_3] = e_1^{l_3}\).

\Cref{lem:quotientbyisolatorAC} tells us that the subgroup \(\langle e_1 \rangle = \sqrt[N]{\gamma_2(N)}\) is characteristic and the quotient \(\Gamma' := \Gamma/\langle e_1\rangle\) is a \(3\)-dimensional crystallographic group. Using \cref{cor:charRoo}, we know that if \(\Gamma'\) has the \(R_\infty\)-property, then so does \(\Gamma\). In \cite{deki96-1,de02-1} the almost-crystallographic groups were classified into families based on which crystallographic group \(\Gamma'\) is. Since only twelve \(3\)-dimensional crystallographic groups do not have the \(R_\infty\)-property, we need only consider the corresponding twelve families of \(4\)-dimensional almost-crystallographic groups. 

Each of these families can be split in smaller subfamilies, determined by the action of \(F\) on \( \sqrt[N]{\gamma_2(N)}\): every \(A \in F\) acts on \(e_1\) by \(^A e_1 = e_1^{\epsilon_A}\) with \(\epsilon_A \in \{-1,1\}\). The following proposition quickly deals with the subfamilies where \(F\) does not act trivially on \(\sqrt[N]{\gamma_2(N)}\).

\begin{proposition}
	\label{prop:r4c2Facttrivially}
	Let \(\Gamma\) be an almost-crystallographic group with translation subgroup \(N\) of rank \(4\) and nilpotency class \(2\), and holonomy group \(F\). If \(F\) acts non-trivially on \(\sqrt[N]{\gamma_2(N)}\), then \(\Gamma\) has the \(R_\infty\)-property.
\end{proposition}

\begin{proof}
	Let \(A \in F\) arbitrary and \(\varphi = \xi_{(d,D)} \in \Aut(\Gamma)\). Since \(A\) acts on \(\langle e_1\rangle = \sqrt[N]{\gamma_2(N)}\) by \(^A e_1 = e_1^{\epsilon_A}\) with \(\epsilon_A \in \{-1,1\}\) and \(\varphi(e_1) = e_1^\nu\) with \(\nu \in \{-1,1\}\), \(A_*\) and \(D_*\) must have the following forms: 
	\begin{equation*}
	A_* = \begin{pmatrix}
	\epsilon_A & *&*&*\\
	0 & *&*&*\\
	0 &* &* &*\\
	0 & * & *&  *\\
	\end{pmatrix},\quad
	D_* = \begin{pmatrix}
	\nu &* &*&*\\
	0 &* &*&*\\
	0 & *& *&*\\
	0 & *&*&*\\
	\end{pmatrix}.
	\end{equation*}
	Thus, \(\I_4-A_*D_*\) is of the form
\begin{equation*}
\I_4-A_*D_* = \begin{pmatrix}
1 - \nu \epsilon_A &* &*&*\\
0 &* &*&*\\
0 & *& *&*\\
0 & *&*&*\\
\end{pmatrix}.
\end{equation*}
Now let us look at specific \(A\in F\). First, let \(A\) be the neutral element of \(F\), which necessarily acts trivially on \(e_1\). The above matrix then has upper left entry \(1-\nu\), hence \(\det(\I_4-D_*) \neq 0\) if and only if \(\nu = -1\).

Second, let \(A\) be an element of \(F\) for which \(\epsilon_A = -1\). Such element exists since we assumed \(F\) acts non-trivially on \(\sqrt[N]{\gamma_2(N)}\). Then the matrix \(\I_4-A_*D_*\) has upper left entry \(1 + \nu\), and \(\det(\I_4-A_*D_*) \neq 0\) if and only if \(\nu = 1\).

As \(\nu\) cannot be \(-1\) and \(1\) at the same time, we always have some \(A \in F\) for which \(\det(\I_4-A_*D_*) = 0\), and by \cref{thm:det1-AD} this means that \(R(\varphi) = \infty\). Since this holds for any automorphism, \(\Gamma\) has the \(R_\infty\)-property.
\end{proof}
From the proof of the theorem above, we can also conclude the following:
\begin{proposition}
	\label{prop:r4c2e1mapstoinverse}
		Let \(\Gamma\) be an almost-crystallographic group with translation subgroup \(N\) of rank \(4\) and nilpotency class \(2\), and let \(e_1\) be a generator of \(\sqrt[N]{\gamma_2(N)}\). If \(\varphi \in \Aut(\Gamma)\) has finite Reidemeister number, then \(\varphi(e_1) = e_1^{-1}\). 
\end{proposition}

We will number the twelve families under consideration according to the crystallographic group \(\Gamma / \sqrt[N]{\gamma_2(N)}\), using the classification in the International Tables in Crystallography \cite{aroy16-1}: they are families \(1\)-\(5\), \(16\), \(19\), \(22\)-\(24\), \(143\) and \(146\). When we write \(\Gamma_{n/m}\), we mean the \(n\)-dimensional crystallographic group with IT-number \(m\).

Using the techniques in \cite[Section 5.4]{deki96-1}, we find that for an almost-crys\-tal\-lo\-gra\-phic group belonging to one of the families \(16\), \(19\) or \(22\)-\(24\), \(F\) acting trivially on \(\sqrt[N]{\gamma_2(N)}\) implies that the group is actually crystallographic. Therefore we may omit these families and we are left with only \(7\) families to study.

Note that the presentations given in this paper may vary from those in \cite{deki96-1,de02-1}. Let \(\Gamma\) and \(\lambda\) denote a group and its faithful representation as given in this paper, and let \(\Gamma'\) and \(\mu\) be the corresponding group and representation as given by \cite{deki96-1} or \cite{de02-1}. \Cref{tbl:conjugacymatricesfamilies} contains a matrix \(\delta\) such that 
\begin{equation*}
\lambda(\Gamma) = \delta \mu(\Gamma')\delta^{-1},
\end{equation*}
hence \(\lambda(\Gamma)\) and \(\mu(\Gamma')\) are conjugate subgroups of \(\Aff(\RR^4)\) and therefore \(\Gamma\) and \(\Gamma'\) are isomorphic.

\begin{table}
	\centering
	\begin{tabular}{c|c}
		Family & \(\delta\)\\
		\hline\\[\dimexpr-\normalbaselineskip+2pt]
		1,2&\(\begin{smallpmatrix}
		1 & 0 & 0 & 0 & 0\\
		0 & 1 & 0 & 0 & 0\\
		0 & 0 & 1 & 0 & 0\\
		0 & 0 & 0 & 1 & 0\\
		0 & 0 & 0 & 0 & 1
		\end{smallpmatrix}\)\\[\dimexpr+12pt]
		3,4&\(\begin{smallpmatrix}
			1 & 0 & 0 & 0 & 0\\
			0 & 0 & 1 & 0 & 0\\
			0 & 1 & 0 & 0 & 0\\
			0 & 0 & 0 & 1 & 0\\
			0 & 0 & 0 & 0 & 1
		\end{smallpmatrix}\)\\[\dimexpr+12pt]
		5&\(\begin{smallpmatrix}
		1 & 0 & 0 & 0 & 0\\
		0 & 1 & 0 & 0 & 0\\
		0 & 1 & 1 & 0 & 0\\
		0 & 0 & 0 & 1 & 0\\
		0 & 0 & 0 & 0 & 1
		\end{smallpmatrix}\)\\[\dimexpr+12pt]
		143&\(\begin{smallpmatrix}
		1 & 0 & 0 & 0 & 0\\
		0 & 0 & 0 & 1 & 0\\
		0 & 1 & 0 & 0 & 0\\
		0 & 0 & 1 & 0 & 0\\
		0 & 0 & 0 & 0 & 1
		\end{smallpmatrix}\)\\[\dimexpr+12pt]
		146&\(\begin{smallpmatrix}
		1 & -\frac{k_1}{2} +k_2+2k_3& -k_2+k_3 & 0 & 0\\
		0 & 1 & 0 & 0 & 0\\
		0 & -1 & 1 & 0 & 0\\
		0 & -1 & 0 & 1 & 0\\
		0 & 0 & 0 & 0 & 1
		\end{smallpmatrix}\)
	\end{tabular}
	\caption{Conjugacy matrices between representations}
	\label{tbl:conjugacymatricesfamilies}
\end{table}

\paragraph{Family 1.}  This family consists of the finitely generated, torsion-free, nilpotent groups of nilpotency class \(2\) and rank \(4\).
It was shown in \cite[Section 3.2]{dtv18-1} that these groups do not have the \(R_\infty\)-property.

\paragraph{Family 2.} Every group in this family has a presentation of the form

\begin{equation*}
\left\langle e_1,e_2,e_3,e_4,\alpha \;\Bigg\rvert \begin{array}{ll}
\comm{e_2}{e_1} = 1 & \alpha e_1 = e_1 \alpha\\
\comm{e_3}{e_1} = 1 & \alpha e_2 = e_1^{k_4}e_2^{-1}\alpha \\
\comm{e_4}{e_1} = 1 & \alpha e_3 =  e_1^{k_5}e_3^{-1}\alpha\\
\comm{e_3}{e_2} = e_1^{k_1}& \alpha e_4 = e_1^{k_6}e_4^{-1}\alpha \\
\comm{e_4}{e_2} = e_1^{k_2} & \alpha^2 = e_1^{k_7}\\
\comm{e_4}{e_3} = e_1^{k_3} & 
\end{array}\right\rangle,
\end{equation*}
and the faithful representation \(\lambda\) is given by
\begin{equation*}
\lambda(\alpha) = \begin{pmatrix}
1 & k_4 & k_5 & k_6 & \frac{k_7}{2}\\
0 & -1 & 0 & 0 & 0\\
0 & 0 & -1 & 0 & 0\\
0 & 0 & 0 & -1 & 0\\
0 & 0 & 0 & 0 & 1
\end{pmatrix}.
\end{equation*}
Set \(k := \gcd(k_1,k_2,k_3)\) and \(g := e_2^{k_3/k}e_3^{-k_2/k}e_4^{k_1/k}\), then the centre \(Z(N)\) of the translation subgroup is generated by \(e_1\) and \(g\). Let \(\varphi: \Gamma \to \Gamma\) be any automorphism. Since \(\langle e_1 \rangle\) and \(Z(N)\) are both characteristic in \(\Gamma\), we have that \(\varphi(g) = g^{\epsilon} e_1^m\) for some \(\epsilon \in \{-1,1\}\) and \(m \in \ZZ\). Consider the induced automorphism \(\varphi' = \xi_{(d',D')}\) on \(\Gamma / \langle e_1 \rangle \cong \Gamma_{3/2}\). Then
\begin{equation*}
\varphi'(g\langle e_1 \rangle) = D'(g\langle e_1 \rangle) = \varphi(g)\langle e_1 \rangle = g^{\epsilon}\langle e_1 \rangle.
\end{equation*}
Depending on the value of \(\epsilon\), \(D'_*\) has either eigenvalue \(1\), in which case \(\det (\I_3 - D'_*) = 0\), or eigenvalue \(-1\), in which case \(\det(\I_3 + D'_*) = 0\). Since the holonomy group of \(\Gamma_{3/2}\) is \(\{\I_3,-\I_3\}\), we obtain by \cref{thm:det1-AD} that \(R(\varphi') = \infty\) and by \cref{lem:diagram3props} that therefore \(R(\varphi) = \infty\). Since this holds for an arbitrary automorphism, \(\Gamma\) has the \(R_\infty\)-property.

\paragraph{Families 3, 4 and 5.} Every group in one of these families has a presentation of the form
\begin{equation*}
\left\langle e_1,e_2,e_3,e_4,\alpha \;\Bigg\rvert \begin{array}{ll}
\comm{e_2}{e_1} = 1 & \alpha e_1 = e_1 \alpha \\
\comm{e_3}{e_1} = 1& \alpha e_2 = e_2\alpha \\
\comm{e_4}{e_1} = 1 & \alpha e_3 = e_1^{k_2}e_2^{-\nu} e_3^{-1}\alpha  \\
\comm{e_3}{e_2} = 1& \alpha e_4 =  e_1^{k_3}e_4^{-1}\alpha\\
\comm{e_4}{e_2} = 1 & \alpha^2 = e_1^{k_4}e_2^{\mu}\\
\comm{e_4}{e_3} = e_1^{k_1} & 
\end{array}\right\rangle,
\end{equation*}
and the faithful representation \(\lambda\) is given by
\begin{equation*}
\lambda(\alpha) = \begin{pmatrix}
1 & 0 & k_2 & k_3 & \frac{k_4}{2}\\
0 & 1 & -\nu & 0 & \frac{\mu}{2}\\
0 & 0 & -1 & 0 & 0 \\
0 & 0 & 0 & -1 & 0\\
0 & 0 & 0 & 0 & 1
\end{pmatrix}.
\end{equation*}
Family \(3\) is given by \(\mu, \nu = 0\), family \(4\) by \(\mu = 1, \nu = 0\) and family \(5\) by \(\mu = 0, \nu = 1\). Define an automorphism \(\varphi = \xi_{(d,D)}\) by
\begin{align*}
&\varphi(e_1) = e_1^{-1},\\
&\varphi(e_2) = e_2^{-1},\\
&\varphi(e_3) = e_1^{k_1 - k_2 - k_3}e_2^\nu e_3e_4^2 ,\\
&\varphi(e_4) = e_1^{3k_1 - k_2 - 2k_3}e_2^\nu e_3^2e_4^3,\\
&\varphi(\alpha) = e_1^{-k_4} e_2^{-\mu}\alpha,
\end{align*}
then \(D_*\) is of the form
\begin{equation*}
D_* = \begin{pmatrix}
-1 & * & * & * \\
0 & -1 & * & *\\
0 & 0 & 1 & 2 \\
0 & 0 & 2 & 3
\end{pmatrix}.
\end{equation*}
We can apply \cref{thm:det1-AD} to show that \(R(\varphi) < \infty\) and hence \(\Gamma\) does not have the \(R_\infty\)-property.

\paragraph{Families 143 and 146.} Every group in one of these families has a presentation of the form
\begin{equation*}
\left\langle e_1,e_2,e_3,e_4,\alpha \;\Bigg\rvert \begin{array}{ll}
\comm{e_2}{e_1} = 1 &\alpha e_1 = e_1 \alpha\\
\comm{e_3}{e_1} = 1&\alpha e_2 = e_2\alpha\\
\comm{e_4}{e_1} = 1&\alpha e_3 = e_1^{k_2}e_4\alpha   \\
\comm{e_3}{e_2} = 1& \alpha e_4 = e_1^{k_3}e_2^\mu e_3^{-1}e_4^{-1}\alpha \\
\comm{e_4}{e_2} = 1 & \alpha^3 = e_1^{k_4}\\
\comm{e_4}{e_3} = e_1^{k_1} & 
\end{array}\right\rangle,
\end{equation*}
and the faithful representation \(\lambda\) is given by
\begin{equation*}
\lambda(\alpha) = \begin{pmatrix}
1 & 0 & k_2 & -\frac{k_1}{2}+k_3  &\frac{k_4}{3}\\
0 & 1 & 0 & \mu & 0\\
0 & 0 & 0 & -1 & 0 \\
0 & 0 & 1 & -1 & 0 \\
0 & 0 & 0 & 0 & 1
\end{pmatrix}.
\end{equation*}
Family \(143\) is given by \(\mu = 0\) and family \(146\) by \(\mu = 1\). Using an argument identical to the proof of \cite[Theorem 4.4, family 13]{dp11-1}, we may conclude that all groups in these families have the \(R_\infty\)-property.

\subsection{Nilpotency class \(3\)}
By an argument analogous to \cite[Example 5.2]{gw09-2}, a finitely-generated, torsion-free, nilpotent group of nilpotency class \(3\) and rank \(4\) has the \(R_\infty\)-property. Applying \cref{cor:charRoo} then proves that every \(4\)-dimensional almost-crystallographic group with translation subgroup of nilpotency class \(3\) has the \(R_\infty\)-property.

\section{The Reidemeister spectra of the 3-dimensional al\-most-crys\-tal\-lo\-gra\-phic groups}

Let \(\Gamma\) be an almost-crystallographic group whose translation subgroup \(N\) is a nilpotent group of rank \(3\) and nilpotency class \(2\). Such \(N\) can be given the following presentation:
\begin{equation*}
\left\langle e_1,e_2,e_3 \;\Big\rvert \begin{array}{ll}
\comm{e_2}{e_1} = 1 & \comm{e_3}{e_2} = e_1^{l_1}\\
\comm{e_3}{e_1} = 1 &  
\end{array}\right\rangle,
\end{equation*}
with \(l_1 > 0\). Moreover, let \(G\) be the Lie group that \(\Gamma\) is modelled on. By \cite[Theorem 4.1]{deki95-1}, there exists a faithful affine representation \(\lambda: G \rtimes \Aut(G) \to \Aff(\RR^3)\) such that its restriction to \(\Gamma\) is again a faithful affine representation. In particular, 
\begin{align*}
 \lambda(e_1) &= \begin{pmatrix}
1 & 0 & 0 & 1\\
0 & 1 & 0 & 0\\
0 & 0 & 1 & 0\\
0 & 0 & 0 & 1
\end{pmatrix},&\quad
\lambda(e_2) &= \begin{pmatrix}
1 & 0 & -\frac{l_1}{2} &0\\
0 & 1 & 0  & 1\\
0 & 0 & 1  & 0\\
0 & 0 & 0  & 1
\end{pmatrix},&\quad
\lambda(e_3) &= \begin{pmatrix}
1 & \frac{l_1}{2}& 0 &0\\
0 & 1 & 0 & 0\\
0 & 0 & 1 & 1\\
0 & 0 & 0 & 1
\end{pmatrix},
\end{align*}
where the value of \(l_1\) is determined by the relation \([e_3,e_2] = e_1^{l_1}\). Like in \cref{sec:2stepnilpclass4}, we have that the subgroup \(\langle e_1 \rangle = \sqrt[N]{\gamma_2(N)}\) is characteristic in \(\Gamma\), and an automorphism \(\varphi\) must satisfy \(\varphi(e_1) = e_1^{-1}\) to have finite Reidemeister number.

As mentioned before, in \cite[Theorem 4.4]{dp11-1} it was shown that there are only \(2\) families of almost-crystallographic groups that do not admit the \(R_\infty\)-property. We again number these families according to the IT-number of the quotient \(\Gamma / \sqrt[N]{\gamma_2(N)}\).

\paragraph{Family 1.}
The groups in this family are exactly the finitely generated, torsion-free, nilpotent groups of nilpotency class \(2\) and rank \(3\). In \cite[Section 3]{roma11-1} it was shown that these groups have Reidemeister spectrum \(2\NN \cup \{\infty\}\). This was shown specifically for the case \(k_1 = 1\), but the argument holds for any \(k_1 > 0\).

\paragraph{Family 2.} Every group in this family has a presentation of the form
\begin{equation*}
\left\langle e_1,e_2,e_3,\alpha \;\Bigg\rvert \begin{array}{ll}
\comm{e_2}{e_1} = 1 &\alpha e_1 = e_1 \alpha\\
\comm{e_3}{e_1} = 1&\alpha e_2 =  e_1^{k_2}e_2^{-1}\alpha\\
\comm{e_3}{e_2} = e_1^{k_1}& \alpha e_3 = e_1^{k_3}e_3^{-1}\alpha \\
&\alpha^2 = e_1^{k_4}
\end{array}\right\rangle,
\end{equation*}
and the faithful representation \(\lambda\) is given by
\begin{equation*}
\lambda(\alpha) = \begin{pmatrix}
1 & k_2 & k_3 & \frac{k_4}{2}\\
0 & -1 & 0 & 0 \\
0 & 0 & -1 & 0\\
0 & 0 & 0 & 1
\end{pmatrix}.
\end{equation*}
Let \(\varphi\) be an automorphism with finite Reidemeister number \(R(\varphi)\). Under the representation \(\lambda\), this automorphism will correspond to a matrix \(\delta \in \Aff(\RR^4)\) such that
\begin{equation*}
\lambda(\varphi(\gamma)) = \delta \lambda(\gamma) \delta^{-1}.
\end{equation*}
for all \(\gamma \in \Gamma\). Since we assumed that \(R(\varphi) < \infty\), we have that \(\varphi(e_1) = e_1^{-1}\). Moreover, \(\varphi\) induces an automorphism \(\varphi'\) on \(\Gamma' := \Gamma/\langle e_1 \rangle\). Thus, \(\delta\) must be of the form
\begin{equation*}
\delta = \begin{pmatrix}
-1 & n_1 & n_2 & 0\\
0 & m_1 & m_3 & d_1/2\\
0 & m_2 & m_4 & d_2/2\\
0 & 0 & 0 & 1
\end{pmatrix},
\end{equation*}
where the constants \(m_i\), \(d_j\) are integers, \(m_1m_4 - m_2m_3 = -1\) and \(n_1, n_2 \in \RR\). Using a computer, one can calculate the (unique) values of \(n_1,n_2\) and \(l_1,l_2,l_3\) such that
\begin{align*}
\delta \lambda(e_2)\delta^{-1} &= \lambda(e_1)^{l_1} \lambda(e_2)^{m_1}\lambda(e_3)^{m_2},\\
\delta \lambda(e_3)\delta^{-1} &= \lambda(e_1)^{l_2} \lambda(e_2)^{m_3}\lambda(e_3)^{m_4},\\
\delta \lambda(\alpha)\delta^{-1} &= \lambda(e_1)^{l_3} \lambda(e_2)^{d_1}\lambda(e_3)^{d_2}\lambda(\alpha).
\end{align*}
From the obtained values of \(l_1\), \(l_2\) and \(l_3\), we get
\begin{align*}
&\varphi(e_1) = e_1^{-1},\\
&\varphi(e_2) = e_1^{\frac{k_1}{2}(m_1m_2 + m_1d_2 - m_2d_1) - \frac{k_2}{2}(m_1+1)-\frac{k_3}{2}m_2}e_2^{m_1}e_3^{m_2},\\
&\varphi(e_3) = e_1^{\frac{k_1}{2}(m_3m_4 + m_3d_2 - m_4d_1) - \frac{k_2}{2}m_3-\frac{k_3}{2}(m_4+1)}e_2^{m_3}e_3^{m_4},\\
&\varphi(\alpha) = e_1^{\frac{k_1}{2}d_1d_2 - \frac{k_2}{2}d_1 - \frac{k_3}{2}d_2 - k_4}e_2^{d_1}e_3^{d_2}\alpha,
\end{align*}
where all exponents must be integers. This places four conditions on the \(m_i\) and \(d_j\):
\begin{enumerate}
	\item[(a)] \( k_1(m_1m_2 + m_1d_2 - m_2d_1) - k_2(m_1+1)-k_3m_2 \equiv 0 \mod 2\),
	\item[(b)] \( k_1(m_3m_4 + m_3d_2 - m_4d_1) - k_2m_3-k_3(m_4+1) \equiv 0 \mod 2\),
	\item[(c)] \( k_1d_1d_2 - k_2d_1 - k_3d_2 \equiv 0 \mod 2\),
	\item[(d)] \( m_1m_4 - m_2m_3 = -1 \).	
\end{enumerate}
For ease of notation, let us set 
\begin{equation*}
M := \begin{pmatrix}
m_1 & m_3\\
m_2 & m_4
\end{pmatrix} \in \GL_2(\ZZ),\quad d := \begin{pmatrix}
d_1\\
d_2
\end{pmatrix} \in \ZZ^2.
\end{equation*}
We will determine \(R(\varphi)\) in a very similar way to the proof of \cite[Proposition 5.11]{dkt17-2}. Let \([x]_{\varphi}\) be a Reidemeister class of \(\Gamma\), then for any \(k \in \ZZ\),
\begin{equation*}
x = (e_1^{-k})x e_1^{2k}\varphi(e_1^{-k})^{-1},
\end{equation*}
therefore \(x \sim_\varphi xe_1^{2k}\) for all \(k \in \ZZ\). Consider the quotient group \(\Gamma' = \Gamma / \langle e_1 \rangle\) and let \(\varphi' = \xi_{(d/2,M)}\) be the induced automorphism on this quotient. Since we assumed that \(R(\varphi) < \infty\), we have that \(R(\varphi') < \infty\) as well. \cite[Proposition 5.10]{dkt17-2} tells us that \(R(\varphi') = |\tr(M)| + O(\I_2 - M,d)\) with
\begin{equation*}
O(A,a) := \# \left\{\bar{x} \in \ZZ_2^2 \mid \bar{A}\bar{x} = \bar{a}\right\},
\end{equation*}
where the bar-notation denotes the element-wise projection to \(\ZZ_2\). 
A Reidemeister class \([x\langle e_1\rangle]_{\varphi'}\) of \(\Gamma'\) will lift to at most \(2\) Reidemeister classes of \(\Gamma\): \([x]_{\varphi}\) and \([xe_1]_{\varphi}\); so the number of lifts is either \(2\) (when \(x \not\sim_\varphi xe_1\)) or \(1\) (when \(x \sim_\varphi xe_1\)). The latter happens if and only if
\begin{equation}
\label{eq:xsimxc}
\exists z \in \Gamma: x e_1 = zx \varphi(z)^{-1}.
\end{equation}
Projecting this to the quotient \(\Gamma'\), we have
\begin{equation}
\label{eq:xsimxcquotient}
\exists z \in \Gamma : x\langle e_1\rangle = zx \varphi(z)^{-1}\langle e_1\rangle.
\end{equation}
Since \(e_1\) is central in \(\Gamma\) and \(x\) appears exactly once on each side of the equality sign in \eqref{eq:xsimxc}, the \(e_1\)-component of \(x\) does not matter. Set \(x = e_2^{x_2}e_3^{x_3}\alpha^{\epsilon_x}\) and \(z = e_1^{z_1}e_2^{z_2}e_3^{z_3}\alpha^{\epsilon_z}\). Let us first assume that \(\epsilon_z = 0\), then \eqref{eq:xsimxcquotient} is equivalent to 
\begin{equation*}
\exists z_2,z_3 \in \ZZ: (\I_2 - AM)\begin{pmatrix}
z_2\\
z_3
\end{pmatrix} = 0,
\end{equation*}
with \(A\) the holonomy part of \(x\langle e_1 \rangle\). As \(R(\varphi') <\infty\), we must have \(z_2 = z_3 = 0\). But then \(z = e_1^{z_1}\), and \eqref{eq:xsimxc} then becomes \(xe_1 = xe_1^{2z_1}\). As \(z_1\) is an integer, this is impossible. So, let us assume that \(\epsilon_z = 1\). Writing out \eqref{eq:xsimxc} component-wise, we find that this condition is equivalent to the following:
\begin{enumerate}
	\item[] There exist \(z_1,z_2,z_3 \in \ZZ\) such that:
	\item[(i)] \(2\begin{pmatrix}
	x_2\\
	x_3
	\end{pmatrix} = \left(\I_2-(-1)^{\epsilon_x}M\right)\begin{pmatrix}
	z_2\\
	z_3
	\end{pmatrix}-(-1)^{\epsilon_x}d\),
	\item[(ii)] \(k_1z_2z_3 - k_2z_2 -k_3z_3-k_4 +1 = 2z_1\).
\end{enumerate}
Condition (i) is independent of the \(e_1\)-components, and hence can be interpreted in terms of the quotient group \(\Gamma'\). In the proof of \cite[Proposition 5.11]{dkt17-2} it was shown that, for a fixed value of \(\epsilon_x\), the number of Reidemeister classes \([x\langle e_1\rangle]_{\varphi'}\) for which a pair \((z_2,z_3)\) satisfying (i) exists is exactly \(O(\I_2-M,d)\), i.e. the number of solutions \((\bar{z}_2,\bar{z}_3) \in \ZZ_2^2\) 	of the linear system of equations
\begin{equation*}
\textrm{(i')\quad} \left(\overline{\I_2-M}\right)\begin{pmatrix}
\bar{z}_2\\
\bar{z}_3
\end{pmatrix} = \bar{d}.
\end{equation*}
Note that the above equation is exactly condition (i) taken modulo \(2\).

Since \(\epsilon_x\) can take two values (\(1\) and \(-1\)), there are in total \(2O(\I_2-M,d)\) Reidemeister classes \([x\langle e_1\rangle]_{\varphi'}\) satisfying condition (i). On the other hand, there are \(|\tr(M)| - O(\I_2-M,d)\) Reidemeister classes of \(\Gamma'\) for which condition (i) does not hold (see \cite[Section 5]{dkt17-2}).

Recall that the variable \(z_1\) appears only in condition (ii). If we have a Reidemeister class \([x\langle e_1\rangle]_{\varphi'}\) and a pair \((z_2,z_3)\) for which (i) holds, then we can find a \(z_1 \in \ZZ\) to make condition (ii) hold if and only if
\begin{equation*}
\textrm{(ii')\quad} \bar{k}_1\bar{z}_2\bar{z}_3 - \bar{k}_2\bar{z}_2 -\bar{k}_3\bar{z}_3-\bar{k}_4 + \bar{1} = \bar{0},
\end{equation*}
which is exactly condition (ii) taken modulo \(2\).

We partition the solutions of (i') into those that do not satisfy condition (ii') and those that do. Let \(S\) be the number of the former and \(T\) the number of the latter, then \(S + T = O(\I_2-M,d)\). Of the  \(2O(\I_2-M,d)\) Reidemeister classes \([x\langle e_1\rangle]_{\varphi'}\)  satisfying condition (i), \(2S\) lift to two distinct Reidemeister classes \([x]_{\varphi}\) and \([xe_1]_{\varphi}\), and \(2T\) lift to a single Reidemeister class \([x]_{\varphi}\). All together, we have
\begin{align*}
R(\varphi) &= 2 (|\tr(M)| - S - T) + 2(2S) + 2T\\
&= 2(|\tr(M)| + S).
\end{align*} 
In particular, we get that \(R(\varphi) \in 2\NN\). Taking the parity of \(\tr(M)\) into account, we can further determine the possible Reidemeister numbers:
\begin{equation*}
R(\varphi) \in \begin{cases}
4\NN + 2S & \text{if }\tr(M) \equiv 0 \pmod 2,\\
4\NN + 2S - 2 & \text{if }\tr(M) \equiv 1 \pmod 2,\\
\end{cases}
\end{equation*}
where
\begin{equation*}
S \leq O(\I_2-M,d) \leq \begin{cases}
	4 & \text{if }\tr(M) \equiv 0 \pmod 2,\\
	1 & \text{if }\tr(M) \equiv 1 \pmod 2.\\
\end{cases}
\end{equation*}
There is one special case, however. If \(M \equiv \I_2 \mod 2\) all entries of \(\I_2 - M\) will be multiples of \(2\); so \(|\det(\I_2-M)| = |\tr(M)| \in 4\NN\) and therefore \(R(\varphi) \in 8\NN + 2S\).

For a fixed group \(\Gamma\) in this family (i.e. a fixed \(4\)-tuple of parameters \((k_1,k_2,k_3,k_4)\)), an automorphism \(\varphi \in \Aut(\Gamma)\) is uniquely determined by the matrix \(M \in \GL_2(\ZZ)\) and the vector \(d \in \ZZ^2\). Our goal is to find out, for each group in the family (or equivalently, for each tuple \((k_1,k_2,k_3,k_4)\)), which \(M\) and \(d\) satisfy conditions (a) - (d) and thus produce an automorphism. 

Conditions (a) - (c) are actually conditions over \(\ZZ_2\), and none of the parameters \(k_i\) appear in condition (d). Therefore, only the parity of the \(k_i\) will play a role, so we need to check \(16\) cases, each corresponding to an element of \(\ZZ_2^4\). Furthermore, a group with parameters \((k_1,k_2,k_3,k_4)\) is isomorphic to the group with parameters \((-k_1,k_3,k_2,k_4)\), which allows us to omit the cases \((0,1,0,0)\), \((0,1,0,1)\), \((1,1,0,0)\) and \((1,1,0,1)\), leaving only \(12\) cases. Rather than trying to find all couples \((M,d)\) (of which there are likely to be infinitely many), we can start by finding all couples \((\bar{M},\bar{d}) \in \GL_2(\ZZ_2) \times \ZZ_2^2\) satisfying conditions (a)-(c).

The function \textsc{MakeList} defined in \cref{alg:spectragammaquotientis2} does exactly this. Moreover, it assigns to every couple a set \(R\), which is the set of possible Reidemeister numbers the corresponding automorphisms can have. The results can be found in tables \ref{tbl:ACfam0000} to \ref{tbl:ACfam1111}. The Reidemeister spectrum of a group is a subset of (or the entirety of) the union of all these sets \(R\).

Next, for each quadruplet of parameters, we tried to find a family of automorphisms whose Reidemeister numbers produce the union of these sets \(R\). We succeeded in this for every choice of parameters, hence the Reidemeister spectrum always equals the union of the \(R\). These automorphisms and their Reidemeister spectra, for all \((k_1,k_2,k_3,k_4)\), can be found in \cref{tbl:ACspectra}. For the sake of brevity, we omitted \(\infty\) from the spectra in this table.

We may thus conclude that, depending on the parity of the parameters \(k_1\), \(k_2\), \(k_3\) and \(k_4\), the Reidemeister spectrum is \(2\NN \cup \{\infty\}\), \(4\NN \cup \{\infty\}\), \((4\NN-2) \cup \{\infty\}\) or \((2\NN +2) \cup \{\infty\}\). Note that all almost-Bieberbach groups have parameters with parities \((0,0,0,1)\) and therefore have spectrum \(2\NN \cup \{\infty\}\).

\begin{algorithm}
	\caption{\textsc{MakeList} function}
	\label{alg:spectragammaquotientis2}
	\begin{algorithmic}[1]
		\Function{MakeList}{$k_1,k_2,k_3,k_4$}
		\Let{AutList}{\(\varnothing\)}
		\For{ \(\bar{M} \in \GL_2(\ZZ_2)\), \(\bar{d} \in \ZZ_2^2\)}
		\If{conditions (1), (2), (3) are met}
		\Let{\(S\)}{0}
		\For{\(\bar{z} \in \ZZ_2^2\)}
		\If{\(\bar{z}\) satisfies (i') but not (ii')}
		\Let{\(S\)}{\(S + 1\)}
		\EndIf
		\EndFor
		\If{\(\tr(M)\equiv 0 \mod 2\)}
		\If{\(M \equiv \I_2 \mod 2\)}
		\Let{\(R\)}{\(8\NN + 2S\)}
		\Else
		\Let{\(R\)}{\(4\NN + 2S\)}
		\EndIf
		\Else
		\Let{\(R\)}{\(4\NN + 2S-2\)}
		\EndIf
		\Let{AutList}{AutList \( \cup \left\{(\bar{M},\bar{d},R)\right\} \)}
		\EndIf
		\EndFor
		\State \Return AutList
		\EndFunction
	\end{algorithmic}
\end{algorithm}

\section{Spectra of 4D almost-Bieberbach groups}

We already determined in \cref{sec:Rinfty4D} which families of four-dimensional al\-most-crys\-tal\-lo\-graphic groups do not have the \(R_\infty\)-property. In \cite{deki96-1} it is determined which groups among these families are almost-Bieberbach groups. We use the presentations from \cref{sec:Rinfty4D}.

\paragraph{Family 1.}  Every group in this family is a finitely generated, torsion-free, nilpotent group of rank \(4\) and nilpotency class \(2\). In \cite[Section 3.2]{dtv18-1} it was shown that the Reidemeister spectrum of such group is always \(4\NN \cup \{\infty\}\).

\paragraph{Family 3.} The almost-Bieberbach groups in this family are those with \((k_1,k_2,k_3,k_4) = (2k,0,0,1)\) for some \(k \in \NN\). An automorphism \(\varphi = \xi_{(d,D)}\) with \(R(\varphi) < \infty\) must be of the form
\begin{align*}
&\varphi(e_1) = e_1^{-1},\\
&\varphi(e_2) = e_1^{l}e_2^{-1},\\
&\varphi(e_3) = e_1^{k(m_1m_2 + m_1d_2 - m_2d_1)}e_3^{m_1}e_4^{m_2},\\
&\varphi(e_4) = e_1^{k(m_3m_4 + m_3d_2 - m_4d_1)}e_3^{m_3}e_4^{m_4},\\
&\varphi(\alpha) = e_1^{kd_1d_2 - 1}e_3^{d_1}e_4^{d_2}\alpha,
\end{align*}
with \(m_1\), \(m_2\), \(m_3\), \(m_4\), \(d_1\), \(d_2\), \(l\) \(\in \ZZ\) and \(m_1m_4 - m_2m_3 = -1\). Then \(D_*\) is of the form
\begin{equation*}
D_* = \begin{pmatrix}
-1 & *  & *& *\\
0 & -1 & * & *\\
0 & 0 & m_1 & m_3\\
0 & 0 & m_2 & m_4
\end{pmatrix}.
\end{equation*}
Using \cref{thm:averagingalmostbieb}, we find that \(R(\varphi) =  4|m_1 + m_4| \in 4\NN\). Now, take the automorphism \(\varphi_m\) given by
\begin{align*}
&\varphi_m(e_1) = e_1^{-1}, &&\varphi_m(e_4) = e_1^{km}e_3e_4^{m},\\
&\varphi_m(e_2) = e_2^{-1}, && \varphi_m(\alpha) = e_1^{- 1}\alpha,\\
&\varphi_m(e_3) = e_4, && 
\end{align*}
with \(m \in \NN\). Then \(R(\varphi_m) = 4m\) and hence \(\Spec_R(\Gamma) = 4\NN \cup \{\infty\}\).

\paragraph{Family 4.} The almost-Bieberbach groups in this family are those where either \((k_1,k_2,k_3,\allowbreak k_4) = (k,0,0,0)\) with \(k \in \NN\) or \((k_1,k_2,k_3,k_4) = (2k,1,0,0)\) with \(k \in \NN\). In the former case, such almost-Bieberbach group can be seen as an internal semidirect product \(H_k \rtimes \ZZ\), where \(H_k = \langle e_1,e_3,e_4\rangle\) and \(\ZZ = \langle \alpha \rangle\). Similarly, in the latter case, a group is an internal semidirect product  \(H_{2k} \rtimes \ZZ\). 

Both of these semidirect products were studied in \cite[Proposition 5.23]{dtv18-1}, their Reidemeister spectra are respectively \(4\NN \cup \{\infty\}\) and \(8\NN \cup \{\infty\}\).

\paragraph{Family 5.} 
The almost-Bieberbach groups in this family are those where \((k_1,k_2,k_3,k_4) = (k,0,0,1)\) with \(k \in \NN\). An automorphism \(\varphi = \xi_{(d,D)}\) with \(R(\varphi) < \infty\) must be of the form
\begin{align*}
&\varphi(e_1) = e_1^{-1},\\
&\varphi(e_2) = e_2^{-1}e_1^{k(2m_1m_2 + 2m_1d_2 - 2m_2d_1 - m_2 - d_2)- 2l},\\
&\varphi(e_3) = e_2^{m_1}e_3^{-1+2m_1}e_4^{m_2}e_1^{l},\\
&\varphi(e_4) = e_2^{m_3}e_3^{2m_3}e_4^{1+2m_4}e_1^{k(2m_3m_4 + m_3d_2 + m_3 - 2m_4d_1 - d_1)},\\
&\varphi(\alpha) = e_2^{d_1}e_3^{2d_1}e_4^{d_2}e_1^{kd_1d_2 - 1}\alpha,
\end{align*}
with \(m_1\), \(m_2\), \(m_3\), \(m_4\), \(d_1\), \(d_2\), \(l\) \(\in \ZZ\) and \(m_1 - m_4 + 2m_1m_4 - m_2m_3 = 0\). Then \(D_*\) is of the form
\begin{equation*}
D_* = \begin{pmatrix}
-1 & * & * & *\\
0 & -1 & * & *\\
0 & 0 & -1+2m_1 & 2m_3\\
0 & 0 & m_2 & 1+2m_4
\end{pmatrix}.
\end{equation*}
Using \cref{thm:averagingalmostbieb}, we find that \(R(\varphi) = 8|m_1+m_4| \in 8\NN \cup \{\infty\}\). Now, take the automorphism \(\varphi_m\) given by
\begin{align*}
&\varphi_m(e_1) = e_1^{-1},&&\varphi_m(e_4) = e_1^{km}e_2^me_3^{2m}e_4,\\
&\varphi_m(e_2) = e_1^{k(2m-1)}e_2^{-1},&&\varphi_m(\alpha) = e_1^{- 1}\alpha,\\
&\varphi_m(e_3) = e_2^me_3^{2m-1}e_4,&&
\end{align*}
with \(m \in \NN\). Then \(R(\varphi_m) = 8m\) and hence \(\Spec_R(\Gamma) = 8\NN \cup \{\infty\}\).

\section{Conclusion}
We have determined which (non-crystallographic) almost-crystallographic groups of dimension \(4\) admit the \(R_\infty\) property, and calculated the Reidemeister spectra of the non-crystallographic 3-dimensional almost-crystallographic groups, as well as the spectra of the non-crystallographic 4-dimensional almost-Bieberbach groups. Together with the results of \cite{dkt17-2}, this completes the calculation of the Reidemeister spectra of the \(3\)-dimensional almost-crystallographic groups and of the \(4\)-dimensional almost-Bieberbach groups.

\paragraph{Acknowledgement} The author would like to thank the referee for their careful reading and useful suggestions for the paper.

\bibliographystyle{plain}
\bibliography{ReidemeisterSpectrum_Tertooy_Sam_AM_Arxiv}

\noindent Sam Tertooy\\
KU Leuven Campus Kulak Kortrijk\\
Etienne Sabbelaan 53\\
8500 Kortrijk\\
Belgium\\[2mm]
Sam.Tertooy\@@kuleuven.be

\begin{table}
	\centering
	\begin{tabular}{c|c|l}
		\(\bar{M}\) & \(\bar{d}\) & \(R\)\\
		\hline
		&&\\[\dimexpr-\normalbaselineskip+2pt]
		\(\begin{smallpmatrix}
		0 & 1\\
		1 & 0
		\end{smallpmatrix}\)&\(\begin{smallpmatrix}
		0 \\ 0
		\end{smallpmatrix}\) &
		\(4\NN+4\)
		\\
		\(\begin{smallpmatrix}
		0 & 1\\
		1 & 0
		\end{smallpmatrix}\)&\(\begin{smallpmatrix}
		0 \\ 1
		\end{smallpmatrix}\) &
		\(4\NN\)
		\\
		\(\begin{smallpmatrix}
		0 & 1\\
		1 & 0
		\end{smallpmatrix}\)&\(\begin{smallpmatrix}
		1 \\ 0
		\end{smallpmatrix}\) &
		\(4\NN\)
		\\
		\(\begin{smallpmatrix}
		0 & 1\\
		1 & 0
		\end{smallpmatrix}\)&\(\begin{smallpmatrix}
		1 \\ 1
		\end{smallpmatrix}\) &
		\(4\NN+4\) 
		\\
		\(\begin{smallpmatrix}
		0 & 1\\
		1 & 1
		\end{smallpmatrix}\)&\(\begin{smallpmatrix}
		0 \\ 0
		\end{smallpmatrix}\) &
		\(4\NN\)
		\\
		\(\begin{smallpmatrix}
		0 & 1\\
		1 & 1
		\end{smallpmatrix}\)&\(\begin{smallpmatrix}
		0 \\ 1
		\end{smallpmatrix}\) &
		\(4\NN\)
		\\
		\(\begin{smallpmatrix}
		0 & 1\\
		1 & 1
		\end{smallpmatrix}\)&\(\begin{smallpmatrix}
		1 \\ 0
		\end{smallpmatrix}\) &
		\(4\NN\) 
		\\
		\(\begin{smallpmatrix}
		0 & 1\\
		1 & 1
		\end{smallpmatrix}\)&\(\begin{smallpmatrix}
		1 \\ 1
		\end{smallpmatrix}\) &
		\(4\NN\) 
		\\
		\(\begin{smallpmatrix}
		1 & 0\\
		0 & 1
		\end{smallpmatrix}\)&\(\begin{smallpmatrix}
		0 \\ 0
		\end{smallpmatrix}\) &
		\(8\NN+8\) 
		\\
		\(\begin{smallpmatrix}
		1 & 0\\
		0 & 1
		\end{smallpmatrix}\)&\(\begin{smallpmatrix}
		0 \\ 1
		\end{smallpmatrix}\) &
		\(8\NN\)
		\\
		\(\begin{smallpmatrix}
		1 & 0\\
		0 & 1
		\end{smallpmatrix}\)&\(\begin{smallpmatrix}
		1 \\ 0
		\end{smallpmatrix}\) &
		\(8\NN\) 
		\\
		\(\begin{smallpmatrix}
		1 & 0\\
		0 & 1
		\end{smallpmatrix}\)&\(\begin{smallpmatrix}
		1 \\ 1
		\end{smallpmatrix}\) &
		\(8\NN\) 
		\\
		\(\begin{smallpmatrix}
		1 & 0\\
		1 & 1
		\end{smallpmatrix}\)&\(\begin{smallpmatrix}
		0 \\ 0
		\end{smallpmatrix}\) &
		\(4\NN+4\) 
		\\
		\(\begin{smallpmatrix}
		1 & 0\\
		1 & 1
		\end{smallpmatrix}\)&\(\begin{smallpmatrix}
		0 \\ 1
		\end{smallpmatrix}\) &
		\(4\NN+4\)
		\\
		\(\begin{smallpmatrix}
		1 & 0\\
		1 & 1
		\end{smallpmatrix}\)&\(\begin{smallpmatrix}
		1 \\ 0
		\end{smallpmatrix}\) &
		\(4\NN\) 
		\\
		\(\begin{smallpmatrix}
		1 & 0\\
		1 & 1
		\end{smallpmatrix}\)&\(\begin{smallpmatrix}
		1 \\ 1
		\end{smallpmatrix}\) &
		\(4\NN\)
		\\
		\(\begin{smallpmatrix}
		1 & 1\\
		0 & 1
		\end{smallpmatrix}\)&\(\begin{smallpmatrix}
		0 \\ 0
		\end{smallpmatrix}\) &
		\(4\NN+4\)
		\\
		\(\begin{smallpmatrix}
		1 & 1\\
		0 & 1
		\end{smallpmatrix}\)&\(\begin{smallpmatrix}
		0 \\ 1
		\end{smallpmatrix}\) &
		\(4\NN\)
		\\
		\(\begin{smallpmatrix}
		1 & 1\\
		0 & 1
		\end{smallpmatrix}\)&\(\begin{smallpmatrix}
		1 \\ 0
		\end{smallpmatrix}\) &
		\(4\NN+4\)
		\\
		\(\begin{smallpmatrix}
		1 & 1\\
		0 & 1
		\end{smallpmatrix}\)&\(\begin{smallpmatrix}
		1 \\ 1
		\end{smallpmatrix}\) &
		\(4\NN\)  
		\\
		\(\begin{smallpmatrix}
		1 & 1\\
		1 & 0
		\end{smallpmatrix}\)&\(\begin{smallpmatrix}
		0 \\ 0
		\end{smallpmatrix}\) &
		\(4\NN\)
		\\
		\(\begin{smallpmatrix}
		1 & 1\\
		1 & 0
		\end{smallpmatrix}\)&\(\begin{smallpmatrix}
		0 \\ 1
		\end{smallpmatrix}\) &
		\(4\NN\)
		\\
		\(\begin{smallpmatrix}
		1 & 1\\
		1 & 0
		\end{smallpmatrix}\)&\(\begin{smallpmatrix}
		1 \\ 0
		\end{smallpmatrix}\) &
		\(4\NN\)
		\\
		\(\begin{smallpmatrix}
		1 & 1\\
		1 & 0
		\end{smallpmatrix}\)&\(\begin{smallpmatrix}
		1 \\ 1
		\end{smallpmatrix}\) &
		\(4\NN\) 
	\end{tabular}
	\caption{\textsc{MakeList}(\(0,0,0,0\))}
	\label{tbl:ACfam0000}
\end{table}

\begin{table}
	\centering
	\begin{tabular}{c|c|l}
		\(\bar{M}\) & \(\bar{d}\) & \(R\)\\
		\hline
		&&\\[\dimexpr-\normalbaselineskip+2pt]
		\(\begin{smallpmatrix}
		0 & 1\\
		1 & 0
		\end{smallpmatrix}\)&\(\begin{smallpmatrix}
		0 \\ 0
		\end{smallpmatrix}\) &
		\(4\NN\)
		\\
		\(\begin{smallpmatrix}
		0 & 1\\
		1 & 0
		\end{smallpmatrix}\)&\(\begin{smallpmatrix}
		0 \\ 1
		\end{smallpmatrix}\) &
		\(4\NN\)
		\\
		\(\begin{smallpmatrix}
		0 & 1\\
		1 & 0
		\end{smallpmatrix}\)&\(\begin{smallpmatrix}
		1 \\ 0
		\end{smallpmatrix}\) &
		\(4\NN\) 
		\\
		\(\begin{smallpmatrix}
		0 & 1\\
		1 & 0
		\end{smallpmatrix}\)&\(\begin{smallpmatrix}
		1 \\ 1
		\end{smallpmatrix}\) &
		\(4\NN\)
		\\
		\(\begin{smallpmatrix}
		0 & 1\\
		1 & 1
		\end{smallpmatrix}\)&\(\begin{smallpmatrix}
		0 \\ 0
		\end{smallpmatrix}\) &
		\(4\NN-2\) 
		\\
		\(\begin{smallpmatrix}
		0 & 1\\
		1 & 1
		\end{smallpmatrix}\)&\(\begin{smallpmatrix}
		0 \\ 1
		\end{smallpmatrix}\) &
		\(4\NN-2\)
		\\
		\(\begin{smallpmatrix}
		0 & 1\\
		1 & 1
		\end{smallpmatrix}\)&\(\begin{smallpmatrix}
		1 \\ 0
		\end{smallpmatrix}\) &
		\(4\NN-2\)
		\\
		\(\begin{smallpmatrix}
		0 & 1\\
		1 & 1
		\end{smallpmatrix}\)&\(\begin{smallpmatrix}
		1 \\ 1
		\end{smallpmatrix}\) &
		\(4\NN-2\)
		\\
		\(\begin{smallpmatrix}
		1 & 0\\
		0 & 1
		\end{smallpmatrix}\)&\(\begin{smallpmatrix}
		0 \\ 0
		\end{smallpmatrix}\) &
		\(8\NN\) 
		\\
		\(\begin{smallpmatrix}
		1 & 0\\
		0 & 1
		\end{smallpmatrix}\)&\(\begin{smallpmatrix}
		0 \\ 1
		\end{smallpmatrix}\) &
		\(8\NN\)
		\\
		\(\begin{smallpmatrix}
		1 & 0\\
		0 & 1
		\end{smallpmatrix}\)&\(\begin{smallpmatrix}
		1 \\ 0
		\end{smallpmatrix}\) &
		\(8\NN\) 
		\\
		\(\begin{smallpmatrix}
		1 & 0\\
		0 & 1
		\end{smallpmatrix}\)&\(\begin{smallpmatrix}
		1 \\ 1
		\end{smallpmatrix}\) &
		\(8\NN\) 
		\\
		\(\begin{smallpmatrix}
		1 & 0\\
		1 & 1
		\end{smallpmatrix}\)&\(\begin{smallpmatrix}
		0 \\ 0
		\end{smallpmatrix}\) &
		\(4\NN\) 
		\\
		\(\begin{smallpmatrix}
		1 & 0\\
		1 & 1
		\end{smallpmatrix}\)&\(\begin{smallpmatrix}
		0 \\ 1
		\end{smallpmatrix}\) &
		\(4\NN\) 
		\\
		\(\begin{smallpmatrix}
		1 & 0\\
		1 & 1
		\end{smallpmatrix}\)&\(\begin{smallpmatrix}
		1 \\ 0
		\end{smallpmatrix}\) &
		\(4\NN\) 
		\\
		\(\begin{smallpmatrix}
		1 & 0\\
		1 & 1
		\end{smallpmatrix}\)&\(\begin{smallpmatrix}
		1 \\ 1
		\end{smallpmatrix}\) &
		\(4\NN\) 
		\\
		\(\begin{smallpmatrix}
		1 & 1\\
		0 & 1
		\end{smallpmatrix}\)&\(\begin{smallpmatrix}
		0 \\ 0
		\end{smallpmatrix}\) &
		\(4\NN\) 
		\\
		\(\begin{smallpmatrix}
		1 & 1\\
		0 & 1
		\end{smallpmatrix}\)&\(\begin{smallpmatrix}
		0 \\ 1
		\end{smallpmatrix}\) &
		\(4\NN\) 
		\\
		\(\begin{smallpmatrix}
		1 & 1\\
		0 & 1
		\end{smallpmatrix}\)&\(\begin{smallpmatrix}
		1 \\ 0
		\end{smallpmatrix}\) &
		\(4\NN\) 
		\\
		\(\begin{smallpmatrix}
		1 & 1\\
		0 & 1
		\end{smallpmatrix}\)&\(\begin{smallpmatrix}
		1 \\ 1
		\end{smallpmatrix}\) &
		\(4\NN\) 
		\\
		\(\begin{smallpmatrix}
		1 & 1\\
		1 & 0
		\end{smallpmatrix}\)&\(\begin{smallpmatrix}
		0 \\ 0
		\end{smallpmatrix}\) &
		\(4\NN-2\) 
		\\
		\(\begin{smallpmatrix}
		1 & 1\\
		1 & 0
		\end{smallpmatrix}\)&\(\begin{smallpmatrix}
		0 \\ 1
		\end{smallpmatrix}\) &
		\(4\NN-2\) 
		\\
		\(\begin{smallpmatrix}
		1 & 1\\
		1 & 0
		\end{smallpmatrix}\)&\(\begin{smallpmatrix}
		1 \\ 0
		\end{smallpmatrix}\) &
		\(4\NN-2\) 
		\\
		\(\begin{smallpmatrix}
		1 & 1\\
		1 & 0
		\end{smallpmatrix}\)&\(\begin{smallpmatrix}
		1 \\ 1
		\end{smallpmatrix}\) &
		\(4\NN-2\) 
	\end{tabular}
	\label{tbl:ACfam0001}
	\caption{\textsc{MakeList}(\(0,0,0,1\))}
\end{table}

\begin{table}
	\centering
	\begin{tabular}{c|c|l}
		\(\bar{M}\) & \(\bar{d}\) & \(R\)\\
		\hline
		&&\\[\dimexpr-\normalbaselineskip+2pt]
		\(\begin{smallpmatrix}
		1 & 0\\
		0 & 1
		\end{smallpmatrix}\)&\(\begin{smallpmatrix}
		0 \\ 0
		\end{smallpmatrix}\) &
		\(8\NN+4\) 
		\\
		\(\begin{smallpmatrix}
		1 & 0\\
		0 & 1
		\end{smallpmatrix}\)&\(\begin{smallpmatrix}
		1 \\ 0
		\end{smallpmatrix}\) &
		\(8\NN\) 
		\\
		\(\begin{smallpmatrix}
		1 & 1\\
		0 & 1
		\end{smallpmatrix}\)&\(\begin{smallpmatrix}
		0 \\ 0
		\end{smallpmatrix}\) &
		\(4\NN+4\) 
		\\
		\(\begin{smallpmatrix}
		1 & 1\\
		0 & 1
		\end{smallpmatrix}\)&\(\begin{smallpmatrix}
		1 \\ 0
		\end{smallpmatrix}\) &
		\(4\NN\) 
		\\
	\end{tabular}
	\label{tbl:ACfam0010}
	\caption{\textsc{MakeList}(\(0,0,1,0\))}
\end{table}
\begin{table}
	\centering
	\begin{tabular}{c|c|l}
		\(\bar{M}\) & \(\bar{d}\) & \(R\)\\
		\hline
		&&\\[\dimexpr-\normalbaselineskip+2pt]
		\(\begin{smallpmatrix}
		1 & 0\\
		0 & 1
		\end{smallpmatrix}\)&\(\begin{smallpmatrix}
		0 \\ 0
		\end{smallpmatrix}\) &
		\(8\NN+4\) 
		\\
		\(\begin{smallpmatrix}
		1 & 0\\
		0 & 1
		\end{smallpmatrix}\)&\(\begin{smallpmatrix}
		1 \\ 0
		\end{smallpmatrix}\) &
		\(8\NN\)
		\\
		\(\begin{smallpmatrix}
		1 & 1\\
		0 & 1
		\end{smallpmatrix}\)&\(\begin{smallpmatrix}
		0 \\ 0
		\end{smallpmatrix}\) &
		\(4\NN\) 
		\\
		\(\begin{smallpmatrix}
		1 & 1\\
		0 & 1
		\end{smallpmatrix}\)&\(\begin{smallpmatrix}
		1 \\ 0
		\end{smallpmatrix}\) &
		\(4\NN+4\) 
		\\
	\end{tabular}
	\label{tbl:ACfam0011}
	\caption{\textsc{MakeList}(\(0,0,1,1\))}
\end{table}

\begin{table}
	\centering
	\begin{tabular}{c|c|l}
		\(\bar{M}\) & \(\bar{d}\) & \(R\)\\
		\hline
		&&\\[\dimexpr-\normalbaselineskip+2pt]
		\(\begin{smallpmatrix}
		0 & 1\\
		1 & 0
		\end{smallpmatrix}\)&\(\begin{smallpmatrix}
		0 \\ 0
		\end{smallpmatrix}\) &
		\(4\NN+4\)
		\\
		\(\begin{smallpmatrix}
		0 & 1\\
		1 & 0
		\end{smallpmatrix}\)&\(\begin{smallpmatrix}
		1 \\ 1
		\end{smallpmatrix}\) &
		\(4\NN\)  
		\\
		\(\begin{smallpmatrix}
		1 & 0\\
		0 & 1
		\end{smallpmatrix}\)&\(\begin{smallpmatrix}
		0 \\ 0
		\end{smallpmatrix}\) &
		\(8\NN+4\)  
		\\
		\(\begin{smallpmatrix}
		1 & 0\\
		0 & 1
		\end{smallpmatrix}\)&\(\begin{smallpmatrix}
		1 \\ 1
		\end{smallpmatrix}\) &
		\(8\NN\)  
		\\
	\end{tabular}
	\label{tbl:ACfam0110}
	\caption{\textsc{MakeList}(\(0,1,1,0\))}
\end{table}
\begin{table}
	\centering
	\begin{tabular}{c|c|l}
		\(\bar{M}\) & \(\bar{d}\) & \(R\)\\
		\hline
		&&\\[\dimexpr-\normalbaselineskip+2pt]
		\(\begin{smallpmatrix}
		0 & 1\\
		1 & 0
		\end{smallpmatrix}\)&\(\begin{smallpmatrix}
		0 \\ 0
		\end{smallpmatrix}\) &
		\(4\NN\)  
		\\
		\(\begin{smallpmatrix}
		0 & 1\\
		1 & 0
		\end{smallpmatrix}\)&\(\begin{smallpmatrix}
		1 \\ 1
		\end{smallpmatrix}\) &
		\(4\NN+4\) 
		\\
		\(\begin{smallpmatrix}
		1 & 0\\
		0 & 1
		\end{smallpmatrix}\)&\(\begin{smallpmatrix}
		0 \\ 0
		\end{smallpmatrix}\) &
		\(8\NN+4\)  
		\\
		\(\begin{smallpmatrix}
		1 & 0\\
		0 & 1
		\end{smallpmatrix}\)&\(\begin{smallpmatrix}
		1 \\ 1
		\end{smallpmatrix}\) &
		\(8\NN\)  
		\\
	\end{tabular}
	\label{tbl:ACfam0111}
	\caption{\textsc{MakeList}(\(0,1,1,1\))}
\end{table}

\begin{table}
	\centering
	\begin{tabular}{c|c|l}
		\(\bar{M}\) & \(\bar{d}\) & \(R\)\\
		\hline
		&&\\[\dimexpr-\normalbaselineskip+2pt]
		\(\begin{smallpmatrix}
		0 & 1\\
		1 & 0
		\end{smallpmatrix}\)&\(\begin{smallpmatrix}
		0 \\ 0
		\end{smallpmatrix}\) &
		\(4\NN+2\)  
		\\
		\(\begin{smallpmatrix}
		0 & 1\\
		1 & 1
		\end{smallpmatrix}\)&\(\begin{smallpmatrix}
		0 \\ 1
		\end{smallpmatrix}\) &
		\(4\NN-2\)  
		\\
		\(\begin{smallpmatrix}
		1 & 0\\
		0 & 1
		\end{smallpmatrix}\)&\(\begin{smallpmatrix}
		0 \\ 0
		\end{smallpmatrix}\) &
		\(8\NN+6\)  
		\\
		\(\begin{smallpmatrix}
		1 & 0\\
		1 & 1
		\end{smallpmatrix}\)&\(\begin{smallpmatrix}
		0 \\ 1
		\end{smallpmatrix}\) &
		\(4\NN+2\)  
		\\
		\(\begin{smallpmatrix}
		1 & 1\\
		0 & 1
		\end{smallpmatrix}\)&\(\begin{smallpmatrix}
		1 \\ 0
		\end{smallpmatrix}\) &
		\(4\NN+2\)  
		\\
		\(\begin{smallpmatrix}
		1 & 1\\
		1 & 0
		\end{smallpmatrix}\)&\(\begin{smallpmatrix}
		1 \\ 0
		\end{smallpmatrix}\) &
		\(4\NN-2\)  
		\\
	\end{tabular}
	\label{tbl:ACfam1000}
	\caption{\textsc{MakeList}(\(1,0,0,0\))}
\end{table}
\begin{table}
	\centering
	\begin{tabular}{c|c|l}
		\(\bar{M}\) & \(\bar{d}\) & \(R\)\\
		\hline
		&&\\[\dimexpr-\normalbaselineskip+2pt]
		\(\begin{smallpmatrix}
		0 & 1\\
		1 & 0
		\end{smallpmatrix}\)&\(\begin{smallpmatrix}
		0 \\ 0
		\end{smallpmatrix}\) &
		\(4\NN+2\) 
		\\
		\(\begin{smallpmatrix}
		0 & 1\\
		1 & 1
		\end{smallpmatrix}\)&\(\begin{smallpmatrix}
		0 \\ 1
		\end{smallpmatrix}\) &
		\(4\NN\) 
		\\
		\(\begin{smallpmatrix}
		1 & 0\\
		0 & 1
		\end{smallpmatrix}\)&\(\begin{smallpmatrix}
		0 \\ 0
		\end{smallpmatrix}\) &
		\(8\NN+2\) 
		\\
		\(\begin{smallpmatrix}
		1 & 0\\
		1 & 1
		\end{smallpmatrix}\)&\(\begin{smallpmatrix}
		0 \\ 1
		\end{smallpmatrix}\) &
		\(4\NN+2\)  
		\\
		\(\begin{smallpmatrix}
		1 & 1\\
		0 & 1
		\end{smallpmatrix}\)&\(\begin{smallpmatrix}
		1 \\ 0
		\end{smallpmatrix}\) &
		\(4\NN+2\)  
		\\
		\(\begin{smallpmatrix}
		1 & 1\\
		1 & 0
		\end{smallpmatrix}\)&\(\begin{smallpmatrix}
		1 \\ 0
		\end{smallpmatrix}\) &
		\(4\NN\) 
		\\
	\end{tabular}
	\label{tbl:ACfam1001}
	\caption{\textsc{MakeList}(\(1,0,0,1\))}
\end{table}

\begin{table}
	\centering
	\begin{tabular}{c|c|l}
		\(\bar{M}\) & \(\bar{d}\) & \(R\)\\
		\hline
		&&\\[\dimexpr-\normalbaselineskip+2pt]
		\(\begin{smallpmatrix}
		0 & 1\\
		1 & 0
		\end{smallpmatrix}\)&\(\begin{smallpmatrix}
		1 \\ 1
		\end{smallpmatrix}\) &
		\(4\NN+2\)  
		\\
		\(\begin{smallpmatrix}
		0 & 1\\
		1 & 1
		\end{smallpmatrix}\)&\(\begin{smallpmatrix}
		1 \\ 0
		\end{smallpmatrix}\) &
		\(4\NN-2\)  
		\\
		\(\begin{smallpmatrix}
		1 & 0\\
		0 & 1
		\end{smallpmatrix}\)&\(\begin{smallpmatrix}
		0 \\ 0
		\end{smallpmatrix}\) &
		\(8\NN+6\)  
		\\
		\(\begin{smallpmatrix}
		1 & 0\\
		1 & 1
		\end{smallpmatrix}\)&\(\begin{smallpmatrix}
		0 \\ 0
		\end{smallpmatrix}\) &
		\(4\NN+2\)  
		\\
		\(\begin{smallpmatrix}
		1 & 1\\
		0 & 1
		\end{smallpmatrix}\)&\(\begin{smallpmatrix}
		1 \\ 0
		\end{smallpmatrix}\) &
		\(4\NN+2\)  
		\\
		\(\begin{smallpmatrix}
		1 & 1\\
		1 & 0
		\end{smallpmatrix}\)&\(\begin{smallpmatrix}
		1 \\ 1
		\end{smallpmatrix}\) &
		\(4\NN-2\)  
		\\
	\end{tabular}
	\label{tbl:ACfam1010}
	\caption{\textsc{MakeList}(\(1,0,1,0\))}
\end{table}
\begin{table}
	\centering
	\begin{tabular}{c|c|l}
		\(\bar{M}\) & \(\bar{d}\) & \(R\)\\
		\hline
		&&\\[\dimexpr-\normalbaselineskip+2pt]
		\(\begin{smallpmatrix}
		0 & 1\\
		1 & 0
		\end{smallpmatrix}\)&\(\begin{smallpmatrix}
		1 \\ 1
		\end{smallpmatrix}\) &
		\(4\NN+2\)  
		\\
		\(\begin{smallpmatrix}
		0 & 1\\
		1 & 1
		\end{smallpmatrix}\)&\(\begin{smallpmatrix}
		1 \\ 0
		\end{smallpmatrix}\) &
		\(4\NN\)  
		\\
		\(\begin{smallpmatrix}
		1 & 0\\
		0 & 1
		\end{smallpmatrix}\)&\(\begin{smallpmatrix}
		0 \\ 0
		\end{smallpmatrix}\) &
		\(8\NN+2\)  
		\\
		\(\begin{smallpmatrix}
		1 & 0\\
		1 & 1
		\end{smallpmatrix}\)&\(\begin{smallpmatrix}
		0 \\ 0
		\end{smallpmatrix}\) &
		\(4\NN+2\)  
		\\
		\(\begin{smallpmatrix}
		1 & 1\\
		0 & 1
		\end{smallpmatrix}\)&\(\begin{smallpmatrix}
		1 \\ 0
		\end{smallpmatrix}\) &
		\(4\NN+2\)  
		\\
		\(\begin{smallpmatrix}
		1 & 1\\
		1 & 0
		\end{smallpmatrix}\)&\(\begin{smallpmatrix}
		1 \\ 1
		\end{smallpmatrix}\) &
		\(4\NN\)  
		\\
	\end{tabular}
	\label{tbl:ACfam1011}
	\caption{\textsc{MakeList}(\(1,0,1,1\))}
\end{table}

\begin{table}
	\centering
	\begin{tabular}{c|c|l}
		\(\bar{M}\) & \(\bar{d}\) & \(R\)\\
		\hline
		&&\\[\dimexpr-\normalbaselineskip+2pt]
		\(\begin{smallpmatrix}
		0 & 1\\
		1 & 0
		\end{smallpmatrix}\)&\(\begin{smallpmatrix}
		0 \\ 0
		\end{smallpmatrix}\) &
		\(4\NN+2\)  
		\\
		\(\begin{smallpmatrix}
		0 & 1\\
		1 & 1
		\end{smallpmatrix}\)&\(\begin{smallpmatrix}
		0 \\ 0
		\end{smallpmatrix}\) &
		\(4\NN\) 
		\\
		\(\begin{smallpmatrix}
		1 & 0\\
		0 & 1
		\end{smallpmatrix}\)&\(\begin{smallpmatrix}
		0 \\ 0
		\end{smallpmatrix}\) &
		\(8\NN+2\)  
		\\
		\(\begin{smallpmatrix}
		1 & 0\\
		1 & 1
		\end{smallpmatrix}\)&\(\begin{smallpmatrix}
		0 \\ 0
		\end{smallpmatrix}\) &
		\(4\NN+2\) 
		\\
		\(\begin{smallpmatrix}
		1 & 1\\
		0 & 1
		\end{smallpmatrix}\)&\(\begin{smallpmatrix}
		0 \\ 0
		\end{smallpmatrix}\) &
		\(4\NN+2\)  
		\\
		\(\begin{smallpmatrix}
		1 & 1\\
		1 & 0
		\end{smallpmatrix}\)&\(\begin{smallpmatrix}
		0 \\ 0
		\end{smallpmatrix}\) &
		\(4\NN\)  
		\\
	\end{tabular}
	\label{tbl:ACfam1110}
	\caption{\textsc{MakeList}(\(1,1,1,0\))}
\end{table}
\begin{table}
	\centering
	\begin{tabular}{c|c|l}
		\(\bar{M}\) & \(\bar{d}\) & \(R\)\\
		\hline
		&&\\[\dimexpr-\normalbaselineskip+2pt]
		\(\begin{smallpmatrix}
		0 & 1\\
		1 & 0
		\end{smallpmatrix}\)&\(\begin{smallpmatrix}
		0 \\ 0
		\end{smallpmatrix}\) &
		\(4\NN+2\)  
		\\
		\(\begin{smallpmatrix}
		0 & 1\\
		1 & 1
		\end{smallpmatrix}\)&\(\begin{smallpmatrix}
		0 \\ 0
		\end{smallpmatrix}\) &
		\(4\NN-2\)  
		\\
		\(\begin{smallpmatrix}
		1 & 0\\
		0 & 1
		\end{smallpmatrix}\)&\(\begin{smallpmatrix}
		0 \\ 0
		\end{smallpmatrix}\) &
		\(8\NN+6\)  
		\\
		\(\begin{smallpmatrix}
		1 & 0\\
		1 & 1
		\end{smallpmatrix}\)&\(\begin{smallpmatrix}
		0 \\ 0
		\end{smallpmatrix}\) &
		\(4\NN+2\)  
		\\
		\(\begin{smallpmatrix}
		1 & 1\\
		0 & 1
		\end{smallpmatrix}\)&\(\begin{smallpmatrix}
		0 \\ 0
		\end{smallpmatrix}\) &
		\(4\NN+2\) 
		\\
		\(\begin{smallpmatrix}
		1 & 1\\
		1 & 0
		\end{smallpmatrix}\)&\(\begin{smallpmatrix}
		0 \\ 0
		\end{smallpmatrix}\) &
		\(4\NN-2\) 
		\\
	\end{tabular}
	\caption{\textsc{MakeList}(\(1,1,1,1\))}
	\label{tbl:ACfam1111}
\end{table}

\begin{table}
	\small
	\centering
	\begin{tabular}{c|c|c|l|l}
		\((k_1,k_2,k_3,k_4)\) &\(M\) & \(d\)& \(R(\varphi)\) & \(\Spec_R(\Gamma)\) \\
		\hline
		&&&&\\[\dimexpr-\normalbaselineskip+2pt]
		\((0,0,0,0)\) & \(\begin{smallpmatrix}
		0 & 1\\
		1 & 2m
		\end{smallpmatrix}\)&\(\begin{smallpmatrix}
		0 \\ 1
		\end{smallpmatrix}\) &
		\(4m\) & \(4\NN\) 
		\\
		\((0,0,0,1)\) & \(\begin{smallpmatrix}
		0 & 1\\
		1 & m
		\end{smallpmatrix}\)&\(\begin{smallpmatrix}
		0 \\ 0
		\end{smallpmatrix}\) &
		\( 2m\) &  \(2\NN \)
		\\
		\((0,0,1,0)\) & \(\begin{smallpmatrix}
		1 & 1\\
		2m & 2m-1
		\end{smallpmatrix}\)&\(\begin{smallpmatrix}
		1 \\ 0
		\end{smallpmatrix}\) &
		\(4m\) & \(4\NN\) 
		\\
		\((0,0,1,1)\) & \(\begin{smallpmatrix}
		1 & 1\\
		2m & 2m-1
		\end{smallpmatrix}\)&\(\begin{smallpmatrix}
		0 \\ 0
		\end{smallpmatrix}\) &
		\(4m\)  &  \(4\NN\)
		\\
		\((0,1,1,0)\) & \(\begin{smallpmatrix}
		0 & 1\\
		1 & 2m
		\end{smallpmatrix}\)&\(\begin{smallpmatrix}
		1 \\ 1
		\end{smallpmatrix}\) &
		\( 4m\) & \(4\NN\)
		\\
		\((0,1,1,1)\) & \(\begin{smallpmatrix}
		0 & 1\\
		1 & 2m
		\end{smallpmatrix}\)&\(\begin{smallpmatrix}
		0 \\ 0
		\end{smallpmatrix}\) &
		\( 4m\) & \(4\NN\)
		\\
		\((1,0,0,0)\) & \(\begin{smallpmatrix}
		0 & 1\\
		1 & 2m-1
		\end{smallpmatrix}\)&\(\begin{smallpmatrix}
		0 \\ 1
		\end{smallpmatrix}\) &
		\(4m-2\) & \(4\NN-2 \)
		\\
		\((1,0,0,1)\) & \(\begin{smallpmatrix}
		1 & 1\\
		m & m-1
		\end{smallpmatrix}\)&\(\begin{smallpmatrix}
		1 \\ 0
		\end{smallpmatrix}\) &
		\(2m+2\)  &\(2\NN+2 \)
		\\
		\((1,0,1,0)\) & \(\begin{smallpmatrix}
		0 & 1\\
		1 & 2m-1
		\end{smallpmatrix}\)&\(\begin{smallpmatrix}
		1 \\ 0
		\end{smallpmatrix}\) &
		\( 4m-2\) & \(4\NN-2\)
		\\
		\((1,0,1,1)\) & \(\begin{smallpmatrix}
		m & 1\\
		1 & 0
		\end{smallpmatrix}\)&\(\begin{smallpmatrix}
		1 \\ 1
		\end{smallpmatrix}\) &
		\( 2m+2\) & \(2\NN+2 \)
		\\
		\((1,1,1,0)\) & \(\begin{smallpmatrix}
		0 & 1\\
		1 & m
		\end{smallpmatrix}\)&\(\begin{smallpmatrix}
		0 \\ 0
		\end{smallpmatrix}\) &
		\(2m+2\) &\(2\NN+2\)
		\\
		\((1,1,1,1)\) & \(\begin{smallpmatrix}
		0 & 1\\
		1 & 2m-1
		\end{smallpmatrix}\)&\(\begin{smallpmatrix}
		0 \\ 0
		\end{smallpmatrix}\) &
		\( 4m-2\) &\(4\NN-2 \)
		\\
	\end{tabular}
	\caption{Automorphisms and Reidemeister spectra and for all \((k_1,k_2,k_3,k_4)\)}
	\label{tbl:ACspectra}
\end{table}
\renewcommand{\arraystretch}{1}
\end{document}